\documentclass[reqno,11pt]{amsart}
\usepackage{amsmath}
\usepackage{amssymb}
\usepackage{tabularx}
\usepackage{enumerate}
\usepackage[dvips]{psfrag,graphicx}
\usepackage{color}
\usepackage{subfigure}
\usepackage{cases}
\usepackage{bbm}

\usepackage{mathrsfs}
\usepackage{empheq}
\usepackage{amsfonts,amssymb,amsmath}
\usepackage{epsfig}

\usepackage{comment}
 
\makeatletter
\@addtoreset{equation}{section}
  
\makeatother
\newtheorem{thm}{Theorem}[section]
\newtheorem{lem}[thm]{Lemma}

\theoremstyle{definition}

\theoremstyle{definition}

\begin{document}
\title[Large time behavior of solutions to a cooperative
model]
{Large time behavior of solutions to a cooperative
model with population flux by attractive transition}
 \thanks{This research was
partially supported by JSPS KAKENHI Grant Number 22K03379.}
\author[R. Kato]{Ryuichi Kato$^\dag$}
\author[K. Kuto]{Kousuke Kuto$^\ddag$}
\thanks{$\dag$ Department of Pure and Applied Mathematics, 
Graduate School of Fundamental Science and Engineering,
Waseda University, 
3-4-1 Ohkubo, Shinjuku-ku, Tokyo 169-8555, Japan.}
\thanks{$\ddag$ Department of Applied Mathematics, 
Waseda University, 
3-4-1 Ohkubo, Shinjuku-ku, Tokyo 169-8555, Japan.}
\thanks{{\bf E-mail:} \texttt{kuto@waseda.jp}}
\date{\today}

\begin{abstract} 
This paper is concerned with 
a diffusive Lotka-Volterra cooperative model
with population flux by attractive transition.
We study the time-global well-posedness 
and the large time behavior of solutions in
a case where the habitat is a bounded convex 
domain and random diffusion rates equal to each other.
A main result shows that when the spatial dimension is less than or equal to 3, under a weaker cooperative condition, a classical solution exists 
globally in time if the initial data belongs to a suitable functional space. 
Furthermore, it is shown that if there exists a positive steady state
and the equal random diffusion rate is sufficiently large, then
all positive solutions asymptotically approach the positive 
steady state as time tends to infinity.
\end{abstract}

\maketitle

\section{Introduction}
In this paper, we are concerned with the following
diffusive Lotka-Volterra cooperative model
with strongly coupled nonlinear diffusion terms:
\begin{equation}\label{para}
\begin{cases}
u_{t}=d_{1}\Delta u+\alpha\nabla\cdot\biggl[
v^{2}\,\nabla\biggl(\dfrac{u}{v}\biggr)\biggr]
+u(a_{1}-b_{1}u+c_{1}v),\ \ \ 
&(x,t)\in\Omega\times (0,T),\vspace{1mm} \\
v_{t}\,=d_{2}\Delta v +\beta\nabla\cdot\biggl[
u^{2}\,\nabla\biggl(\dfrac{v}{u}\biggr)\biggr]
+v(a_{2}+b_{2}u-c_{2}v),\ \ \ 
&(x,t)\in\Omega\times (0,T), \vspace{1mm} \\
\dfrac{\partial u}{\partial\nu}=
\dfrac{\partial v}{\partial\nu}=0,\ \ \ 
&(x,t)\in\partial\Omega\times (0,T),\vspace{1mm} \\
u(x,0)=u_{0}(x)\ge 0,\ \ 
v(x,0)=v_{0}(x)\ge 0,\ \ \ 
&x\in\Omega,
\end{cases}
\end{equation}  
where $\Omega\,(\subset\mathbb{R}^{N})$
is a bounded domain with smooth boundary $\partial\Omega$, 
corresponding to the habitat of two species in a symbiotic relationship
with each other. 
The unknown functions $u(x,t)$ and $v(x,t)$ 
represent the population densities 
of the two species
at location $x\in\Omega$ and time $t>0$. 
The homogeneous Neumann boundary conditions will be imposed so
that 
each species has zero derivative in the direction of 
the outward unit normal vector $\nu$ on the boundary 
$\partial\Omega$ of the habitat $\Omega$.
The coefficients $a_i$ $(i=1,2)$ 
in reaction terms of the Lotka-Volterra type are
real constants that can be negative, 
whereas $b_i$ and $c_i$ are positive constants. 
The coefficients $a_i$ are corresponding to the growth rate of each species, 
and $b_1$ and $c_2$ are determined from 
the carrying capacity of environment 
for each species. Here $c_1$ and $b_2$ represent 
the degree of interaction due to the symbiosis between different species. 
Various examples of such symbiotic pairs of species are known,
e.g. 
the symbiotic relationship between ants and aphids is often cited.
The coefficient $d_i$
$(i=1,2)$ in the linear part of diffusion terms 
correspond to the the random diffusion rate of each individual
of the species.
The strongly coupled nonlinear diffusion terms 
$\nabla\cdot [v^{2}\nabla (u/v)]$ and
$\nabla\cdot [u^{2}\nabla (v/u)]$ are
expressed as in the following forms:
\begin{equation}\label{fd}
\begin{split}
&
\nabla\cdot\biggl[
v^{2}\,\nabla\biggl(\dfrac{u}{v}\biggr)\biggr]=
\nabla\cdot
(v\nabla u-u\nabla v)=v\Delta u-u\Delta v, \\
&
\nabla\cdot\biggl[
u^{2}\,\nabla\biggl(\dfrac{v}{u}\biggr)\biggr]=
\nabla\cdot
(u\nabla v-v\nabla u)=u\Delta v-v\Delta u.
\end{split}
\end{equation}
These terms
describe an ecological situation where
each individual of species migrates with higher probability 
to location with higher density of the other species.
From the perspective of the diffusion process in ecology,
such strongly coupled diffusion terms in \eqref{fd} 
microscopically model
a situation where the transition probability of
each individual of each species depends on the density
of the other species at the point of arrival
(Okubo and Levin
\cite[Section 5.4]{OL}).
When modeling the diffusion process of symbiotic relationship,
it makes sense to set such a population flux by attractive transition
to both species.

In terms of the mechanisms of ecological diffusion process, 
the nonlinear diffusion as in \eqref{fd} is as typical 
as so-called cross-diffusion $\Delta (uv)$ and 
chemotaxis $\nabla\cdot (u\nabla v)$. 
Actually, in the well-known monograph \cite{OL}, the above three types 
of diffusion processes are described in parallel. 
In terms of partial differential equations, 
the cross-diffusion and the chemotaxis have been studied pure mathematically 
through the Shigesada-Kawasaki-Teramoto model and the Keller-Segel model, 
respectively, while the nonlinear diffusion as in \eqref{fd}
has not been studied pure mathematically much. 
In this sense, the purpose of this paper is to 
form the basis for mathematical prescriptions 
to deal with the nonlinear diffusion through the analysis of 
the cooperative model \eqref{para}.

The following are studies from the viewpoint of partial differential equations 
for the strongly coupled nonlinear diffusion term such as \eqref{fd}:
For the predator-prey model with 
$\beta=0$ and  $b_{2}<0$ in \eqref{para},
Heihoff and Yokota \cite{HY} established the global solvability in time 
if $N\le 3$ and derived
the convergence of solutions to a positive steady state as $t\to\infty$
with some additional conditions.
For the predator-prey model, the stationary problem was also discussed 
by Oeda and the second author \cite{OK, KO} 
under homogeneous Dirichlet boundary conditions. 
For \eqref{para} itself, Adachi and the second author \cite{AK} 
obtained the bifurcation structure of 
the non-constant positive steady states and 
the asymptotic behavior of the steady states as
$\alpha$, $\beta\to\infty$ 
under a weak cooperative condition.

Based on the existing studies mentioned above, 
in this paper we study the solvability of \eqref{para}. 
As for the time-local solvability, 
we can resort to the unique existence theorem for 
time-local classical solutions 
established by the series of papers by Amann 
\cite{Am1, Am2, Am3, Am4} 
that 
considered the solvability of a class of 
quasilinear parabolic equations
including \eqref{para}.
Actually, by virtue of \eqref{fd}, the parabolic equations
of \eqref{para} are represented as the following divergence form:
$$
\biggl[
\begin{array}{c}
u_{t}\\
v_{t}
\end{array}
\biggr]
=\sum^{N}_{j=1}
\dfrac{\partial}{\partial x_{j}}\biggl(
A(u,v)\dfrac{\partial}{\partial x_{j}}
\biggl[
\begin{array}{c}
u\\
v
\end{array}
\biggr]\biggr)
+
\biggl[
\begin{array}{c}
u(a_{1}-b_{1}u+c_{1}v)\\
v(a_{2}+b_{2}u-c_{2}v)
\end{array}
\biggr],
$$
and all the eigenvalues of
$$
A(u, v):=
\biggl[
\begin{array}{cc}
d_{1}+\alpha v & -\alpha u\\
-\beta v & d_{2}+\beta u
\end{array}
\biggr]
$$
are contained in $\{\,\lambda\in\mathbb{C}\,:\,\mbox{Re}\,\lambda>0\,\}$
for any nonnegative $u$ and $v$.
Then it is possible to check all the conditions
for use of \cite[Theorem on p.17]{Am3} to 
know the following unique existence of
time-local classical solutions of \eqref{para}:
\begin{thm}[\cite{Am3}]\label{Amannthm}
Assume that $u_{0}$ and $v_{0}$ are nonnegative functions belonging to
$W^{1,\infty}(\Omega)$.
Then there exist $T_{{\rm max}}\in (0,\infty]$ and nonnegative 
$$(u,v)\in [\,C(\overline{\Omega}\times [0,T_{{\rm max}}))\cap
C^{2,1}(\overline{\Omega}\times (0,T_{{\rm max}}))\,]^{2}
$$
which solves \eqref{para} classically in
$\overline{\Omega}\times (0,T_{{\rm max}})$.
Furthermore,
if $T_{{\rm max}}<\infty$, then
\begin{equation}\label{Amanncri}
\limsup_{t\nearrow T_{{\rm max}}}\left(
\|u(\,\cdot\,,t)\|_{W^{1,p}(\Omega )}+
\|v(\,\cdot\,,t)\|_{W^{1,p}(\Omega )}\right)=
\infty
\quad\mbox{for all}\ p>N.
\end{equation}
\end{thm}
Once the time-local well-posedness has been established,
as in Theorem \ref{Amannthm}, one is naturally interested
in the time-global solvability.
It should be noted that concerning 
cooperative models such as \eqref{para}, 
even in the corresponding ODE system
with $d_1=d_2 =\alpha =\beta=0$, 
there exist solutions that blow up in finite times 
under the strong cooperative condition $b_{1}c_{2}<b_{2}c_{1}$.
With homogeneous Neumann boundary conditions, 
such solutions of the ODE system
are also solutions of \eqref{para}.
Then it follows that
\eqref{para} possesses finite time
blow-up solutions when $b_{1}c_{2}<b_{2}c_{1}$. 
On the other hand, 
under the weak cooperative condition 
$b_{1}c_{2}>b_{2}c_{1}$, 
it is known that all positive solutions of 
the linear diffusive system with $\alpha=\beta=0$
are global in time and uniformly bounded 
(\cite{Go}, \cite{We}).
It is natural to ask whether
even with the strongly coupled diffusion terms as \eqref{fd}, 
all the classical solutions of 
\eqref{para} still exist globally in time
under the same weak 
cooperative condition.

The first main result of this paper 
will show that under a {\it weaker} cooperative condition involving
$\gamma:=\alpha /\beta$
(which leads to $b_{1}c_{2}>b_{2}c_{1}$), 
if $d_{1}=d_{2}$ and the habitat $\Omega$ is a bounded convex domain  
of spatial dimension $3$ or less, 
then all classical solutions exist globally in time.
Furthermore, in this case, the existence of
an $L^{\infty}(\Omega )$ absorbing set will be shown.
From the point of view of Theorem \ref{Amannthm}, 
for any solution $(u,v)$,
if it can be shown that 
$\|u(\,\cdot\,,t)\|_{W^{1,4}(\Omega )}+\|v(\,\cdot\,,t)\|_{W^{1,4}(\Omega)}$ 
remains bounded for $t\in (0,T_{{\rm max}})$, 
the time-global solvability of the solution 
(i.e. $T_{{\rm max}}=\infty$) can be obtained.
In order to establish the $W^{1,4}(\Omega )$ a priori estimate, 
we first derive the $L^{\infty}(\Omega )$ a priori estimate by
the comparison argument for a semilinear parabolic
equation with unknown function $w:=u+\gamma v$, and then,
combine this $L^{\infty}(\Omega )$ estimate with the 
maximum regularity and the energy estimate for $u$ and $v$.
For the energy estimate, we refer to ideas 
developed by Lou and Winkler \cite{LW} 
in a priori estimates of all the solutions to 
the Shigesada-Kawasaki-Teramoto model with equal random diffusion.

In the weak cooperative case where
$b_{1}c_{2}>b_{2}c_{1}$, 
it is known that
all positive solutions of
the linear diffusive system with $\alpha=\beta=0$
converge to the positive steady state
as $t\to\infty$ if it exists (e.g. \cite{We}).
Under the same conditions as in the first main result, the second
main result will show that if the positive constant steady state 
$(u^{*}, v^{*})$ exists, 
then it is globally asymptotically stable in the sense that
all positive solutions of \eqref{para}
asymptotically approach $(u^{*}, v^{*})$ as $t\to\infty$
if the equal random diffusion rate $d\,(\,=d_{1}=d_{2})$ 
is sufficiently large.
The proof is based on the construction of a Lyapunov function.
On the other hand, if $d$ is sufficiently small with the above setting, 
$(u^{*}, v^{*})$ can no longer be globally asymptotically stable, 
since the result by \cite{AK} implies the existence of 
nonconstant positive steady states 
due to the effect of the strongly coupled nonlinear diffusion terms. 

This paper is organized as follows: 
Section 2 presents the main results. 
Section 3 proves the first main result, which ensures
the existence of the $L^{\infty}(\Omega )$
absorbing set of all positive solutions. 
Section 4 proves the second main result, 
which states the global asymptotic stability of 
the positive constant steady state. 
Section 5 discusses the relevance of the global bifurcation structure 
of positive steady states obtained by \cite{AK} 
to the behavior of the nonstationary solutions obtained in
Sections 3 and 4.

\section{Main results}
In this section, we state main results of this paper.
The first result asserts that in any convex domain $\Omega$ 
of space $3$ dimensions or less, 
all the solutions guaranteed by Theorem \ref{Amannthm} 
exist globally in time, i.e., $T_{{\rm max}}=\infty$, 
with equal random diffusion rates and a weaker cooperative condition. 
Moreover, we also obtain the existence of an 
$L^{\infty}(\Omega )$ absorbing set:
\begin{thm}\label{1stthm}
Let $\Omega$ be a bounded convex domain in $\mathbb{R}^{N}$ with $N\le 3$.
Assume that $d_{1}=d_{2}$ and
\begin{equation}\label{weaker} 
2\sqrt{\gamma b_{1}c_{2}}>c_{1}+\gamma b_{2}\quad\mbox{with}\quad
\gamma=\dfrac{\alpha}{\beta}.
\end{equation}
If nonnegative $u_{0}$ and
$v_{0}$ belong to $W^{1,\infty}(\Omega )$,
then 
$T_{{\rm max}}=\infty$ in Theorem 1.1,
that is,
\eqref{para}
admits a time-global solution
$$
(u,v)\in [\,C(\overline{\Omega}\times [0,\infty))\cap
C^{2,1}(\overline{\Omega}\times (0,\infty))\,]^{2}.
$$
Furthermore,
there exist $K=K(a_{i}, b_{i}, c_{i},\gamma )>0$
and
$t_{0}=t_{0}(\|u_{0}+\gamma v_{0}\|_{L^{\infty}(\Omega )})>0$
such that
\begin{equation}\label{abs}
t\ge t_{0}\ \Longrightarrow\
\|u(\,\cdot\,,t)\|_{L^{\infty}(\Omega )}+\gamma
\|v(\,\cdot\,,t)\|_{L^{\infty}(\Omega )}\le K.
\end{equation}
\end{thm}
It is noted that
\eqref{weaker} leads to
the weak cooperative condition $b_{1}c_{2}>b_{2}c_{1}$
since $c_{1}+\gamma b_{2}\ge \sqrt{\gamma b_{2}c_{1}}$.
In this sense,
the condition \eqref{weaker} involving $\gamma$ is an even
weaker cooperative than the usual weak cooperative condition
$b_{1}c_{2}>b_{2}c_{1}$.

It is easy to check that,
in the weak cooperative case where $b_{1}c_{2}>b_{2}c_{1}$,
\eqref{para} admits a unique positive constant solution
\begin{equation}\label{const}
(u^{*}, v^{*}):=\biggl(
\dfrac{a_{1}c_{2}+a_{2}c_{1}}{b_{1}c_{2}-b_{2}c_{1}},
\dfrac{a_{1}b_{2}+a_{2}b_{1}}{b_{1}c_{2}-b_{2}c_{1}}\biggr)
\end{equation}
provided that
either of the following three conditions holds:
\begin{equation}\label{weakconst}
\arraycolsep=1.6pt
\def\arraystretch{1.5}
\begin{array}{rll}
\text{(i) } & 0<a_{1},\,a_{2};\\
\text{(ii) } & a_{2}<0<a_{1}\quad\mbox{and}\quad
\dfrac{b_{1}}{b_{2}}<
\dfrac{a_{1}}{|a_{2}|};\\
\text{(iii) }& a_{1}<0<a_{2}\quad\mbox{and}\quad
\dfrac{|a_{1}|}{a_{2}}<
\dfrac{c_{1}}{c_{2}}.
\end{array}
\end{equation}
The next result asserts that 
under the same conditions as Theorem \ref{1stthm}, 
if there is a positive constant steady state
$(u^{*}, v^{*})$ and the 
equal random diffusion rate is sufficiently large, then $(u^{*}, v^{*})$ is globally asymptotically stable
(GAS):
\begin{thm}\label{2ndthm}
Let $d_{1}=d_{2}\,(\,=:d\,)$
and $\Omega$ be a bounded convex domain in $\mathbb{R}^{N}$
with $N\le 3$.
Assume \eqref{weaker} and either of (i)-(iii) in \eqref{weakconst},
Then there exists $\overline{d}=\overline{d}(\alpha,\beta,a_{i},b_{i},c_{i})>0$ such that if $d\ge\overline{d}$, then 
any positive solution $(u, v)$ of 
\eqref{para}
satisfies
\begin{equation}\label{L2aym}
\lim_{t\to\infty}(u(\,\cdot\,,t), v(\,\cdot\,,t))=
(u^{*}, v^{*})
\quad\mbox{in}\quad L^{2}(\Omega)\times L^{2}(\Omega ).
\end{equation}
\end{thm}
Here we refer to some results on the existence of 
nonconstant positive steady states found by Adachi and
the second author \cite{AK}.
According to an interpretation of 
\cite[Theorem 1]{AK}
for the case where $d_{1}=d_{2}\,(\,=:d)$
and either of (ii) or (iii) in \eqref{weakconst} is imposed,
if some additional conditions involving $(\alpha, \beta )$ are assumed,
then there exists a positive sequence 
$\{d^{(j)}_{*}\}$ with $d^{(j)}_{*}\to 0$
$(j\to\infty)$ such that
nonconstant positive steady states
bifurcate from $(u^{*}, v^{*})$ at 
$d=d^{(j)}_{*}$
(see Theorem \ref{AKthm1} below
in the case where (ii) of \eqref{weakconst} 
is assumed).
Then,
in a range of small $d>0$,
infinitely many bifurcation points appear on the branch
$\{(d,u^{*}, v^{*})\}$ of the positive constant steady state
from which nonconstant steady states bifurcate
due to the effect of $\alpha$ and $\beta$. 
Hence for such a small range of $d$, 
$(u^{*}, v^{*})$ is no longer globally asymptotically stable.
This scenario is totally different from that of
the weak cooperative linear diffusive system with $\alpha =\beta =0$
in \eqref{para}, where all positive solutions tend to
$(u^{*}, v^{*})$ as $t\to\infty$.

\section{Existence of the $L^{\infty}(\Omega )$ absorbing set}
This section is devoted to the proof of Theorem \ref{1stthm}.

\subsection{A priori estimates for $w=u+\gamma v$}
Taking an advantage of the equal random diffusion rates, 
we first show the $L^{\infty}(\Omega )$ a priori estimate for 
$w=u+\gamma v$ using the comparison theorem for 
a semilinear heat equation which
$w$ satisfies.
\begin{lem}\label{comlem}
Let $d_{1}=d_{2}$ and assume \eqref{weaker}.
For any
nonnegative functions
$u_{0}$,
and
$v_{0}$
belong to $W^{1,\infty}(\Omega )$,
let 
$(u,v)\in [\,C(\overline{\Omega}\times [0,T_{{\rm max}}))\cap
C^{2,1}(\Omega\times (0,T_{{\rm max}}))\,]^{2}$
be the solution of \eqref{para}, which is ensured by
Theorem \ref{Amannthm}.
Then,
there exist
a positive function $\xi (t)$ uniformly bounded in $[0,\infty )$
and a positive constant $\lambda^{*}=\lambda^{*}(b_{i}, c_{i},\gamma )$
such that
$\xi (0)=\|u_{0}+\gamma v_{0}\|_{L^{\infty }(\Omega )}$,
$$w(x,t):=u(x,t)+\gamma v(x,t)\le \xi (t)\quad
\mbox{for all}\ (x, t)\in\Omega\times (0,T_{{\rm max}}),$$
and
$$
\lim_{t\to\infty}\xi (t)=\dfrac{(1+\gamma^{2})\widetilde{a}}{\lambda^{*}},
$$
where
$\widetilde{a}:=\max\{a_{1}^{+},a_{2}^{+}\}$
for $a_{i}^{+}=\max\{a_{i},0\}$ 
$(i=1,2)$.
\end{lem}

\begin{proof}
We set $d:=d_{1}=d_{2}$ and $\gamma:=\alpha/\beta$.
By adding $\gamma $ times of the second equation 
to the first equation in \eqref{para}, 
we find that $w=u+\gamma v$ 
satisfies the following semilinear heat equation:
\begin{equation}\label{wsemi}
\begin{split}
w_{t}&=d\Delta w+u(a_{1}-b_{1}u+c_{1}v)+\gamma v(a_{2}+b_{2}u-c_{2}v)\\
&=
d\Delta w+a_{1}u+\gamma a_{2}v-
[u,v]
\left[
\begin{array}{cc}
b_{1} & -\frac{c_{1}+\gamma b_{2}}{2}\\
-\frac{c_{1}+\gamma b_{2}}{2} & \gamma c_{2}
\end{array}
\right]\left[
\begin{array}{c}
u\\
v
\end{array}
\right].
\end{split}
\end{equation}
Here we note that the matrix
\begin{equation}\label{Adef}
A:=\left[
\begin{array}{cc}
b_{1} & -\frac{c_{1}+\gamma b_{2}}{2}\\
-\frac{c_{1}+\gamma b_{2}}{2} & \gamma c_{2}
\end{array}
\right]
\end{equation}
satisfies
\begin{equation}
\begin{split}
\det A
&=
\gamma b_{1}c_{2}-\biggl(\dfrac{c_{1}+\gamma b_{2}}{2}\biggr)^{2}\\
&=
\biggl(\sqrt{\gamma b_{1}c_{2}}
+\dfrac{c_{1}+\gamma b_{2}}{2}\biggr)
\biggl(\sqrt{\gamma b_{1}c_{2}}
-\dfrac{c_{1}+\gamma b_{2}}{2}\biggr)>0
\end{split}
\nonumber
\end{equation}
by \eqref{weaker}.
Then $A$ is positive definite and
there exists 
$
\lambda^{*}=\lambda^{*}(b_{i}, c_{i}, \gamma )>0
$
$(i=1,2)$
such that
\begin{equation}\label{lamdef}
\boldsymbol{X}A\boldsymbol{X}^{T}\ge \lambda^{*}
|\boldsymbol{X}|^{2}\quad
\mbox{for all}\
\boldsymbol{X}=(
X_{1},
X_{2}
) \in\mathbb{R}^{2}.
\end{equation}
Then it follows from \eqref{wsemi} that
\begin{equation}\label{wsemi2}
\begin{split}
w_{t}&
\le d\Delta w+a_{1}u+\gamma a_{2}v-\lambda^{*}(u^{2}+v^{2})\\
&\le
d\Delta w+w\biggl(\widetilde{a}-\dfrac{\lambda^{*}}{1+\gamma^{2}}w
\biggr)
\quad\mbox{for all}\ (x,t)\in \Omega\times (0,T_{{\rm max}}),
\end{split}
\end{equation}
where 
$\widetilde{a}:=\max\{a_{1}^{+},a_{2}^{+}\}$
for $a_{i}^{+}=\max\{a_{i},0\}$ 
$(i=1,2)$.
Let $\xi (t)$ be the solution to the following initial value problem 
of the logistic equation:
$$
\begin{cases}
\dfrac{d}{dt}\xi (t)=\xi (t)\biggl(\widetilde{a}-
\dfrac{\lambda^{*}}{1+\gamma^{2}}\xi(t)\biggr),\quad t>0,\\
\xi (0)=\|u_{0}+\gamma v_{0}\|_{L^{\infty}(\Omega )}.
\end{cases}
$$
Hence it follows that
$$
\lim_{t\to\infty}\xi (t)=\dfrac{(1+\gamma^{2})\widetilde{a}}{\lambda^{*}}.
$$
Owing to \eqref{wsemi2},
the usual comparison argument implies that
$$
0\le w(x,t)\le \xi (t)\quad\mbox{for all}\
(x, t)\in \Omega\times (0,T_{{\rm max}}).
$$
Then the proof of Lemma \ref{comlem} is complete.
\end{proof}
The following lemma is based on the maximal regularity theory
for \eqref{wsemi} and will help in the derivation of 
the $W^{1,4}(\Omega )$ a priori estimate of $u$ and $v$.
\begin{lem}\label{maxlem}
Assume $d_{1}=d_{2}$ and \eqref{weaker}. 
For
nonnegative functions
$u_{0}$
and
$v_{0}$
belonging to $W^{1,\infty}(\Omega )$, 
let 
$(u,v)\in [\,C(\overline{\Omega}\times [0,T_{{\rm max}}))\cap
C^{2,1}(\Omega\times (0,T_{{\rm max}}))\,]^{2}$
be the solution of \eqref{para}, which is ensured by
Theorem \ref{Amannthm}.
Then there exists a positive constant $\widehat{C}$
such that
$w=u+\gamma v$ satisfies
$$\int^{t+\tau}_{t}\|\Delta w(\,\cdot\,,s)\|_{L^{3}(\Omega)}^{3}\,ds
\le \widehat{C}\quad\mbox{for all}\ t\in (\tau, \widehat{T}_{{\rm max}}),$$
where
\begin{equation}\label{taudef}
\tau:=
\begin{cases}
1\quad&\mbox{if}\ T_{{\rm max}}\ge 3,\\
T_{\rm max}/3\quad&\mbox{if}\ T_{{\rm max}}<3
\end{cases}
\qquad\mbox{and}\qquad
\widehat{T}_{{\rm max}}=
\begin{cases}
T_{{\rm max}}-\tau\quad&\mbox{if}\ T_{{\rm max}}<\infty,\\
\infty &\mbox{if}\ T_{{\rm max}}=\infty
\end{cases}
\end{equation}
and $\widehat{C}$ depends on 
$\|u(\,\cdot\,,\tau )+\gamma v(\,\cdot\,,\tau)\|_{W^{2,3}(\Omega )}$
but is independent of
$t\in (\tau, \widehat{T}_{{\rm max}})$.
\end{lem}

\begin{proof}
For any solution $(u,v)$ of \eqref{para}, we set
$$
h(x,t):=u(x,t)(a_{1}-b_{1}u(x,t)+c_{1}v(x,t))
+\gamma v(x,t)(a_{2}+b_{2}u(x,t)-c_{2}v(x,t)).$$
Then $w=u+\gamma v$ satisfies 
$$
\begin{cases}
w_{t}=d\Delta w+h(x,t),\quad &(x,t)\in\Omega\times (0,T_{{\rm max}}),\\
\dfrac{\partial w}{\partial\nu}=0,\quad
&(x,t)\in\partial\Omega\times  (0,T_{{\rm max}}).
\end{cases}
$$
Under the assumption of Lemma \ref{maxlem},
we know from Lemma \ref{comlem} that
\begin{equation}\label{fub}
\|h(\,\cdot\,,t)\|_{L^{\infty}(\Omega )}\le
C_{0}
\end{equation}
with some positive constant $C_{0}$ independently of $t\in (0,T_{{\rm max}})$.
By usual application of the maximum regularity theory
(e.g.\,\cite {GS}) for
the above heat equation,
one can find a positive constant $\widehat{C}_{1}=\widehat{C}_{1}(p,q)$
for any $p$, $q\in (1,\infty)$ such that
$\widehat{C}_{1}$ is independent of $t\in(\tau,\widehat{T}_{{\rm max}})$ and
\begin{equation}\label{MR}
\begin{split}
&\int^{t+\tau}_{t}
\left(
\|w(\,\cdot\,,s)\|^{q}_{W^{2,p}(\Omega )}
+\|w_{t}(\,\cdot\,,s)\|^{q}_{L^{p}(\Omega )}
\right)\,ds\\
&\le
\widehat{C}_{1}\biggl(
\|u(\,\cdot\,,\tau )+\gamma v(\,\cdot\,,\tau)\|^{q}_{W^{2,p}(\Omega )}+
\int^{t+\tau}_{t}
\|h(\,\cdot\,,s)\|_{L^{p}(\Omega )}^{q}\,ds\biggr)
\quad\mbox{for all}
\ t\in (\tau, \widehat{T}_{{\rm max}}).
\end{split}
\end{equation}
By setting $p=q=3$,
we derive the desired estimate
from \eqref{fub} and \eqref{MR},
\end{proof}

\subsection{H\"older estimate}
It has already been shown that $u$ and $v$ are uniformly bounded on 
$\Omega\times (0,T_{{\rm max}})$ under the assumptions of 
Lemma \ref{comlem}.
The following lemma gives the H\"older estimate of $u$ and $v$. 
This H\"older estimate will play a role of a condition for application of 
an Ehring-type inequality when deriving the
$W^{1,4}(\Omega )$ a priori estimate later.

\begin{lem}\label{Hollem}
Assume $d_{1}=d_{2}$ and \eqref{weaker}.
For
nonnegative functions
$u_{0}$,
and
$v_{0}$
belonging to $W^{1,\infty}(\Omega )$,
let 
$(u,v)\in [\,C(\overline{\Omega}\times [0,T_{{\rm max}}))\cap
C^{2,1}(\overline{\Omega}\times (0,T_{{\rm max}}))\,]^{2}$
be the solution of \eqref{para}, which is ensured by
Theorem \ref{Amannthm}.
Then there exist 
$\theta\in (0,1)$ and $\widehat{C}>0$ independently of
$t\in (\tau, \widehat{T}_{{\rm max}})$ such that
\begin{equation}\label{HolC}
\|u\|_{C^{\theta, \theta /2}(\overline{\Omega}\times [t, t+\tau ])}+
\|v\|_{C^{\theta, \theta /2}(\overline{\Omega}\times [t, t+\tau ])}\le 
\widehat{C}
\quad\mbox{for all}\ t\in (\tau,\widehat{T}_{{\rm max}}),
\end{equation}
where
$\tau $ and $\widehat{T}_{{\rm max}}$ are defined by
\eqref{taudef}.
\end{lem}

\begin{proof}
In the following proof of Lemma \ref{Hollem},
we denote by $C_{i}$, $\widetilde{C}_{i}$ and
$\widehat{C}_{i}$ positive constants
that are independent of $t\in (0,T_{{\rm max}})$,
$t\in (\tau, T_{{\rm max}})$
and
$t\in (\tau, \widehat{T}_{{\rm max}})$,
respectively.
It should be noted that the proof relies on that of
\cite[Lemma 4.13]{LW}.
By setting $q=p\in (1,\infty )$ in \eqref{MR}, we see 
$$
\int^{t+\tau}_{t}
\left(
\|w(\,\cdot\,,s)\|^{p}_{W^{2,p}(\Omega )}
+\|w_{t}(\,\cdot\,,s)\|^{p}_{L^{p}(\Omega )}
\right)\,ds\le\widehat{C}_{1}\quad\mbox{for all}\ 
t\in (\tau, \widehat{T}_{{\rm max}})
$$
with some $\widehat{C}_{1}>0$.
By the embedding theorem \cite[Theorem 5.2]{Am4}
with suitably large $p$, one can find $\widetilde{C}_{1}$ such that
\begin{equation}\label{nw}
\|\nabla w(\,\cdot\, ,t)\|_{L^{\infty}(\Omega )}\le 
\widetilde{C}_{1}
\quad\mbox{for all}\ t\in (\tau, T_{{\rm max}}).
\end{equation}
Obviously, Lemma \ref{comlem} 
yields $C_{1}$ and $C_{2}$ such that
\begin{equation}\label{uv}
u(x,t)\le C_{1},\quad
v(x,t)\le C_{2}
\quad\mbox{for all}\ 
(x,t)\in \Omega\times (0, T_{{\rm max}}).
\end{equation}
In order to apply the H\"older regularity theory
for a class of quasilinear parabolic equations
by Porzio and Vespri \cite[Theorem 1.3 and Remark 1.4]{PV},
we note the following divergence form by substituting
\eqref{fd} and $\nabla v=(\nabla w-\nabla u)/\gamma$ into the first equation of \eqref{para}:
\begin{equation}\label{udiv}
\begin{split}
u_{t}&=
\nabla\cdot (\,d\,\nabla u+\alpha v\nabla u-\alpha u\nabla v\,)+
u(a_{1}-b_{1}u+c_{1}v)\\
&=
\nabla\cdot (\,d\,\nabla u+\alpha v\nabla u-\beta u\nabla w+\beta u\nabla u\,)
+
u(a_{1}-b_{1}u+c_{1}v)\\
&=
\nabla\cdot\{\,(\,d+\beta u+\alpha v\,)\nabla u-\beta u\nabla w\,\}
+
f(u,v),
\end{split}
\end{equation}
where $f(u,v):=u(a_{1}-b_{1}u+c_{1}v)$.
In view of \eqref{udiv},
we define
an $\mathbb{R}^{N}$ valued function 
$\boldsymbol{A}(x,t,\xi)$ and
a real valued function $B(x,t)$ as follows:
\begin{equation}
\begin{split}
\boldsymbol{A}(x,t,\xi):=
(d+\beta u(x,t) +\alpha v(x,t))\xi-
\beta u(x,t)\nabla w(x,t),
\quad&(x,t,\xi)\in\Omega\times
(0, T_{{\rm max}})\times\mathbb{R}^{N},\\
B(x,t):=f(u(x,t), v(x,t)),\quad
&(x,t)\in\Omega\times (0,T_{{\rm max}}).
\end{split}
\nonumber
\end{equation}
Hence \eqref{udiv} can be represented as
$$
u_{t}=\nabla\cdot\boldsymbol{A}(x,t,\nabla u)+B(x,t).
$$
Let us check the structural conditions for use of
\cite[Theorem 1.3 and Remark 1.4]{PV}.
By \eqref{nw}, \eqref{uv} and the Young inequality,
we have
\begin{equation}\label{stu1}
\begin{split}
\boldsymbol{A}(x,t,\nabla u)\cdot\nabla u&=
(d+\beta u+\alpha v)|\nabla u|^{2}-\beta u\nabla u\cdot\nabla w\\
&\ge
(d+\beta u+\alpha v)|\nabla u|^{2}-\beta u|\nabla u||\nabla w|\\
&\ge
(d+\alpha v)|\nabla u|^{2}-\dfrac{\beta}{4}u|\nabla w|^{2}\\
&\ge
d|\nabla u|^{2}-\dfrac{\beta}{4}C_{1}\widetilde{C}_{1}^{2}
\qquad\mbox{for all}\ 
(x,t)\in \Omega\times (\tau, T_{{\rm max}}).
\end{split}
\end{equation}
Furthermore it follows from \eqref{nw} and \eqref{uv} that
\begin{equation}\label{stu2}
|\boldsymbol{A}(x,t,\nabla u)|\le
(d+\beta C_{1}+\alpha C_{2})|\nabla u|+\beta C_{1}\widetilde{C}_{1}
\end{equation}
and
\begin{equation}\label{stu3}
|B(x,t)|\le C_{1}(a_{1}^{+}+c_{1}C_{2})
\end{equation}
for all $(x,t)\in\Omega\times (\tau,T_{{\rm max}})$.
Owing to the structural conditions
\eqref{stu1}, \eqref{stu2} and \eqref{stu3},
one can use \cite[Theorem 1.3 and Remark 1.4]{PV} to
find $\theta\in (0,1)$ and $\widehat{C}_{2}$
such that
$$\|u\|_{C^{\theta,\theta /2}(\overline{\Omega}\times [t,t+\tau ])}\le 
\widehat{C}_{2}
\qquad\mbox{for all}\ t\in (\tau, \widehat{T}_{{\rm max}}).
$$
In the same manner, we can also get the desired H\"older estimate for $v$.
\end{proof}

\subsection{$W^{1,4}(\Omega )$ estimate}
We have seen that under the conditions $d_1=d_2$ and \eqref{weaker}, 
any nonnegative solution $(u,v)$ of \eqref{para} 
remains bounded in the sense that 
$\|u(\,\cdot\,,t)\|_{L^{\infty}(\Omega )}+\|v(\,\cdot\,,t)\|_{L^{\infty}(\Omega )}$ never blows up as $t \nearrow T_{{\rm max}}$.
It should be noted that this fact alone does not allow us to conclude
$T_{{\rm max}} = \infty$.
This is because, although the maximum values of $u$ and $v$
remain bounded as $t\nearrow T_{{\rm max}}$, 
the possibility of blowing up of any of partial derivatives of $u$ or $v$ 
below the second order cannot be ruled out.
In view of Theorem \ref{Amannthm}, we recall that
if $T_{{\rm max}}<\infty$, then
$\|u(\,\cdot\,,t)\|_{W^{1,p}(\Omega )}+\|v(\,\cdot\,,t)\|_{W^{1,p}(\Omega )}$
blows up as $t\nearrow T_{{\rm max}}$ for any 
$p > N$.  Therefore, focusing on its counterpart, 
if it can be shown that $\|u(\,\cdot\,,t)\|_{W^{1,p}(\Omega)}+\|v(\,\cdot\,,t)\|_{W^{1,p}(\Omega)}$ remains bounded as $t\nearrow T_{\rm max}$
with some $p>N$, then
we can conclude not only
$T_{{\rm max}}=\infty$ 
but also what the solution 
$(u,v)$ becomes a time-global classical solution.

Therefore, in this subsection, we show 
the uniform $W^{1,4}(\Omega )$ a priori estimate for any 
nonnegative solution $(u,v)$ with additional conditions that
$N\le 3$ and $\Omega $ is convex.
To this end, we note a useful inequality of the Gronwall type 
described in \cite[Lemma 3.4]{SSW}.

\begin{lem}[\cite{SSW}]\label{Grlem}
Let $T>0$ and $y(t)\in [0,\infty)$ be an
absolutely
continuous function for $t\in [0,T)$ satisfying
\begin{equation}\label{Gr}
\dfrac{d}{dt}y(t)+ay(t)\le h(t)\quad\mbox{for a.e.}\
t\in (0,T)
\end{equation}
with some $a>0$ and a nonnegative 
measurable function
$h (t)$ for which
there exists $b>0$ such that
$$
\int^{t+1}_{t}h(s)\,ds\le b
\quad\mbox{for all}\ t\in [0,T-1).$$
Then
\begin{equation}\label{ybdd}
y(t)\le\max\biggl\{
y(0)+b,
\dfrac{b}{a}+2b\biggr\}
\quad\mbox{for all}\ t\in (0,T).
\end{equation}
\end{lem}
Hereafter we derive \eqref{Gr}
for 
$y(t):=\|\nabla u(\,\cdot\,,t)\|_{L^4(\Omega)}^{4}
+\|\nabla v(\,\cdot\,,t)\|_{L^4(\Omega)}^{4}
$
with some $h(t)$
for the boundedness of $y(t)$ on $t\in (0,T_{{\rm max}})$ as \eqref{ybdd}.
The proof is inspired by the idea of Lou and Winkler \cite[Lemma 5.2]{LW} 
on the construction of time-global solutions of 
the Shigesada-Kawasaki-Teramoto model.

\begin{lem}\label{W14lem}
Let $\Omega$ be a bounded convex domain in $\mathbb{R}^{N}$
with $N\le 3$.
Assume $d_{1}=d_{2}\,(\,=:d)$ and \eqref{weaker}. 
For
nonnegative 
$u_{0}$,
$v_{0}\in W^{1,\infty}(\Omega )$,
let 
$(u,v)\in [\,C(\overline{\Omega}\times [0,T_{{\rm max}}))\cap
C^{2,1}(\overline{\Omega}\times (0,T_{{\rm max}}))\,]^{2}$
be the solution of \eqref{para}.
Then there exists $C>0$ such that
$C$ is independent of $t\in (0,T_{{\rm max}})$ and
$$
\|u(\,\cdot\,,t)\|_{W^{1,4}(\Omega )}+
\|v(\,\cdot\,,t)\|_{W^{1,4}(\Omega )}\le C
\quad\mbox{for all}\ t\in (0,T_{{\rm max}}).$$
\end{lem}

\begin{proof}
In the proof, $x$ and the dot symbol for the spatial variable 
and $t$ for the time variable are omitted 
unless there is confusion.
Substituting 
the first equation of \eqref{para} into
$$
\dfrac{1}{4}\dfrac{d}{dt}
\|\nabla u\|^{4}_{L^{4}(\Omega )}
=
\displaystyle\int_{\Omega}|\nabla u|^{2}
\nabla u\cdot
\nabla u_{t},
$$
we get
\begin{equation}\label{energy}
\begin{split}
&\dfrac{1}{4}\dfrac{d}{dt}
\|\nabla u\|^{4}_{L^{4}(\Omega )}
-d\int_{\Omega}|\nabla u|^{2}\nabla u\cdot\nabla\Delta u\\
=\,&
\alpha\int_{\Omega}|\nabla u|^{2}\nabla u\cdot
\nabla (v\Delta u-u\Delta v)
+\int_{\Omega}|\nabla u|^{2}\nabla u\cdot
\nabla f(u,v),
\end{split}
\end{equation}
where $f(u,v)=u(a_{1}-b_{1}u+c_{1}v)$.
Substituting the identity
$$\nabla u\cdot\nabla\Delta u=\dfrac{1}{2}
\Delta |\nabla u|^{2}-|D^{2}u|^{2}$$ 
into the
random diffusion term gives
$$-d\int_{\Omega}|\nabla u|^{2}\nabla u\cdot\nabla\Delta u=
d\,\biggl(
-\dfrac{1}{2}\int_{\Omega}|\nabla u|^{2}\Delta|\nabla u|^{2}
+\int_{\Omega}|\nabla u|^{2}|D^{2}u|^{2}\biggr).
$$
We integrate by parts in the first term of the right-hand side 
to get
$$
\int_{\Omega}|\nabla u|^{2}
\Delta|\nabla u|^{2}=
\int_{\partial\Omega}
|\nabla u|^{2}\dfrac{\partial |\nabla u|^{2}}{\partial\nu}
-\int_{\Omega}
|\nabla |\nabla u|^{2}|^{2}.
$$
Since  $\partial |\nabla u|^{2}/\partial\nu\le 0$
on $\partial\Omega$ when $\Omega $ is convex
(e.g., \cite[Lemma 1.1]{Li}), then
$$-d\int_{\Omega}|\nabla u|^{2}\nabla u\cdot\nabla\Delta u
\ge
d\,
\int_{\Omega}|\nabla u|^{2}|D^{2}u|^{2}.
$$
Then, together with integration by parts in the first term of
the right-hand side of \eqref{energy}, we obtain
\begin{equation}\label{energy2}
\dfrac{1}{4}\dfrac{d}{dt}
\|\nabla u\|^{4}_{L^{4}(\Omega )}
+d\,
\int_{\Omega}|\nabla u|^{2}|D^{2}u|^{2}
\le
I(t)+J(t)+\int_{\Omega}|\nabla u|^{2}\nabla u\cdot\nabla f(u,v),
\end{equation}
where
$$
I:=
-\alpha\int_{\Omega} (v\Delta u-u\Delta v)|\nabla u|^{2}\Delta u,\qquad
J:=
-\alpha\int_{\Omega} (v\Delta u-u\Delta v)\nabla|\nabla u|^{2}\cdot
\nabla u.
$$
Substituting 
$\Delta v=(\Delta w-\Delta u)/\gamma$
and $\gamma v+u=w$
into $I$ gives
\begin{equation}\label{Iexp}
I=-\beta
\int_{\Omega}w|\nabla u|^{2}|\Delta u|^{2}+
\beta
\int_{\Omega}u|\nabla u|^{2}\Delta u\Delta w.
\end{equation}
Substituting 
$\Delta v=(\Delta w-\Delta u)/\gamma$
into $J$ gives
$$
J=\int_{\Omega }
(-\alpha v\Delta u+\beta u\Delta w-\beta u\Delta u)
\nabla |\nabla u|^{2}\cdot\nabla u.
$$
Owing to the facts that $\partial |\nabla u|^{2}/\partial\nu\le 0$
on $\partial\Omega$ and $\nabla \Delta u\cdot\nabla u=
\frac{1}{2}\Delta |\nabla u|^{2}-|D^{2}u|^{2}$, 
straight forward calculations using integration by parts lead to
\begin{equation}
\begin{split}
J&\le\,
\alpha\int_{\Omega}\,\biggl(\,v|\nabla u|^{2}|\Delta u|^{2}+
|\nabla u|^{2}(\nabla u\cdot\nabla v)\Delta u-
\dfrac{1}{2}|\nabla u|^{2}\nabla v\cdot\nabla |\nabla u|^{2}\,
\biggr)\\
&
+\beta\int_{\Omega}
u\Delta w\,\nabla |\nabla u|^{2}\cdot\nabla u\\
&+\beta\int_{\Omega}\biggl(
\,u|\nabla u|^{2}|\Delta u|^{2}+
|\nabla u|^{4}\Delta u-
\dfrac{1}{2}|\nabla u|^{2}\nabla u\cdot\nabla |\nabla u|^{2}\,
\biggr).
\end{split}
\nonumber
\end{equation}
Since $\beta u+\alpha v=\beta w$, we see
\begin{equation}\label{Jexp}
\begin{split}
J&\le
\beta
\int_{\Omega} w|\nabla u|^{2}|\Delta u|^{2}
+
\alpha\int_{\Omega}\biggl(|\nabla u|^{2}(\nabla u\cdot\nabla v)\Delta u-
\dfrac{1}{2}\,|\nabla u|^{2}\nabla v\cdot\nabla |\nabla u|^{2}
\biggr)\\
&
+\beta\int_{\Omega}\biggl(
u\Delta w\,\nabla |\nabla u|^{2}\cdot\nabla u
+
|\nabla u|^{4}\Delta u-
\dfrac{1}{2}\,|\nabla u|^{2}\nabla u\cdot\nabla |\nabla u|^{2}
\biggr).
\end{split}
\end{equation}
Adding \eqref{Iexp} and \eqref{Jexp} cancels each other's first terms 
in the right-hand sides and results in
\begin{equation}
\begin{split}
&I+J\\
&\le
\beta\int_{\Omega }u|\nabla u|^{2}\Delta u\Delta w+
\alpha\int_{\Omega}|\nabla u|^{2}(\nabla u\cdot\nabla v)\Delta u
-\dfrac{\alpha}{2}\int_{\Omega }|\nabla u|^{2}\nabla v\cdot\nabla
|\nabla u|^{2}\\
&+\beta\int_{\Omega }u\Delta w\,\nabla |\nabla u|^{2}\cdot\nabla u+
\beta\int_{\Omega }|\nabla u|^{4}\Delta u-\dfrac{\beta }{2}\int_{\Omega}
|\nabla u|^{2}\nabla u\cdot\nabla |\nabla u|^{2}\\
&=:I_{1}+J_{2}+J_{3}+J_{4}+J_{5}+J_{6}.
\end{split}
\nonumber
\end{equation}
By virtue of $|\Delta u|\le\sqrt{N}|D^{2}u|$ and
$\nabla|\nabla u|^{2}=2(\nabla u)(D^{2} u)$,
the Young and the H\"older inequalities yield
\begin{equation}
\begin{split}
&I_{1}
\le 
\dfrac{d}{2}\displaystyle\int_{\Omega}|\nabla u|^{2}|D^{2}u|^{2}+
\dfrac{N\beta^{2}\|u\|_{L^{\infty}(\Omega )}^{2}}{2d}\left(
\|\nabla u\|^{6}_{L^{6}(\Omega )}+\|\Delta w\|_{L^{3}(\Omega )}^{3}\right),
\\
&J_{2}\le
\dfrac{d}{4}\displaystyle\int_{\Omega}|\nabla u|^{2}|D^{2}u|^{2}+
\dfrac{N\alpha^{2}}{d}\left(\|\nabla u\|_{L^{6}(\Omega )}^{6}+
\|\nabla v\|_{L^{6}(\Omega )}^{6}\right),
\\
&J_{3}\le
\dfrac{d}{8}\displaystyle\int_{\Omega}|\nabla u|^{2}|D^{2}u|^{2}+
\dfrac{2\alpha^{2}}{d}\left(\|\nabla u\|_{L^{6}(\Omega )}^{6}+
\|\nabla v\|_{L^{6}(\Omega )}^{6}\right),
\vspace{1mm}\\
&J_{4}\le
\dfrac{d}{16}\displaystyle\int_{\Omega}|\nabla u|^{2}|D^{2}u|^{2}+
\dfrac{16\beta^{2}\|u\|_{L^{\infty}(\Omega )}^{2}}{d}\left(\|\nabla u\|_{L^{6}(\Omega )}^{6}+
\|\Delta w\|_{L^{3}(\Omega )}^{3}\right),
\\
&J_{5}\le
\dfrac{d}{32}\displaystyle\int_{\Omega}|\nabla u|^{2}|D^{2}u|^{2}+
\dfrac{8N\beta^{2}}{d} \|\nabla u\|_{L^{6}(\Omega )}^{6},
\\
&J_{6}\le
\dfrac{d}{64}\displaystyle\int_{\Omega}|\nabla u|^{2}|D^{2}u|^{2}+
\dfrac{16\beta^{2}}{d}\|\nabla u\|_{L^{6}(\Omega )}^{6}
\end{split}
\nonumber
\end{equation}
for all $t\in (0,T_{{\rm max}})$.
Recalling Lemma \ref{comlem}, we note
\begin{equation}\label{Linf}
\|u\|_{L^{\infty}(\Omega )},\ \ 
\gamma \|v\|_{L^{\infty}(\Omega )}\le
\max\biggl\{
\|u_{0}+\gamma v_{0}\|_{L^{\infty}(\Omega )},\,
\dfrac{(1+\gamma^{2})\widetilde{a}}{\lambda^{*}}\biggr\}
\end{equation}
for all $t\in (0,T_{{\rm max}})$.
In what follows in the proof,
$C_{i}$ represent positive constants depending on
$(d,\alpha, \beta, a_{i}, b_{i}, c_{i}, N, |\Omega|, \|u_{0}+\gamma v_{0}\|_{L^{\infty}(\Omega )})$.
It is possible to find $C_{1}$ such that
$$\displaystyle\int_{\Omega}|\nabla u|^{2}\nabla u\cdot\nabla f(u,v)
+\|\nabla u\|^{4}_{L^{4}(\Omega )}
\le C_{1}(\|\nabla u\|_{L^{6}(\Omega )}^{6}+
\|\nabla v\|_{L^{6}(\Omega )}^{6}+1).$$
Therefore, we add $\|\nabla u\|^{4}_{L^{4}(\Omega )}$
in both sides of \eqref{energy2} to obtain
\begin{equation}
\begin{split}
&\dfrac{1}{4}\dfrac{d}{dt}
\|\nabla u\|^{4}_{L^{4}(\Omega )}
+\|\nabla u\|^{4}_{L^{4}(\Omega )}
+\dfrac{d}{64}
\int_{\Omega}|\nabla u|^{2}|D^{2}u|^{2}\\
&\le
C_{2}
\left(
\|\nabla u\|^{6}_{L^{6}(\Omega )}+
\|\nabla v\|^{6}_{L^{6}(\Omega )}+
\|\Delta w\|^{3}_{L^{3}(\Omega )}+1
\right)
\quad\mbox{for all}\ t\in (0,T_{{\rm max}})
\end{split}
\nonumber
\end{equation}
with some $C_{2}$.
Together with the same procedure to the second equation of 
\eqref{para}, we find $C_{3}$ such that
\begin{equation}\label{energy3}
\begin{split}
&\dfrac{1}{4}\dfrac{d}{dt}
\left(\|\nabla u\|^{4}_{L^{4}(\Omega )}
+\|\nabla v\|^{4}_{L^{4}(\Omega )}\right)
+\|\nabla u\|_{L^{4}(\Omega )}^{4}+
\|\nabla v\|_{L^{4}(\Omega )}^{4}
+\dfrac{d}{64}\int_{\Omega }
(|\nabla u|^{2}|D^{2}u|^{2}+
|\nabla v|^{2}|D^{2}v|^{2})\\
&\le
C_{3}
\left(
\|\nabla u\|^{6}_{L^{6}(\Omega )}+
\|\nabla v\|^{6}_{L^{6}(\Omega )}+
\|\Delta w\|^{3}_{L^{3}(\Omega )}+1
\right)
\quad\mbox{for all}\ t\in (0,T_{{\rm max}}).
\end{split}
\end{equation}
Thanks to Lemma \ref{Hollem},
we can use
the Ehrling-type lemma \cite[Lemma 5.1]{LW} to obtain 
\begin{equation}\label{Ehrling}
\begin{split}
&
\|\nabla u\|^{6}_{L^{6}(\Omega )}
\le\dfrac{d}{64}\int_{\Omega}|\nabla u|^{2}
|D^{2}u|^{2}+C_{4}\|u\|^{6}_{L^{\infty}(\Omega )},\\
&
\|\nabla v\|^{6}_{L^{6}(\Omega )}
\le\dfrac{d}{64}\int_{\Omega}|\nabla v|^{2}
|D^{2}v|^{2}+C_{4}\|v\|^{6}_{L^{\infty}(\Omega )}
\end{split}
\end{equation}
with some $C_{4}$.
From \eqref{Linf}, \eqref{energy3} and \eqref{Ehrling},
we find $C_{5}$ such that
$$y(t):=\|\nabla u(\,\cdot\,,t)\|^{4}_{L^{4}(\Omega )}+
\|\nabla v(\,\cdot\,,t)\|^{4}_{L^{4}(\Omega )}$$
satisfies
$$
\dfrac{1}{4}\dfrac{d}{dt}y(t)+y(t)\le C_{5}(
\|\Delta w\|^{3}_{L^{3}(\Omega )}+1)\quad
\mbox{for all}\quad t\in (0,T_{{\rm max}}).
$$
By Lemma \ref{maxlem}, there exists $C_{6}$ such that
$$
\int^{t+\tau}_{t}
\|\Delta w(\,\cdot\,,s)\|^{3}_{L^{3}(\Omega )}\,ds\le
C_{6}
\quad\mbox{for all}\ t\in 
(\tau,\widehat{T}_{{\rm max}}),
$$
where $\tau $ and $\widehat{T}_{{\rm max}}$
are defined by \eqref{taudef}.
By Lemma \ref{Grlem},
there exists $C_{7}$ such that
$$
y(t)=\|\nabla u(\,\cdot\,,t)\|^{4}_{L^{4}(\Omega )}+
\|\nabla v(\,\cdot\,,t)\|^{4}_{L^{4}(\Omega )}\le C_{7}
\quad\mbox{for all}\ t\in (\tau, T_{{\rm max}}).
$$
By virtue of Theorem \ref{Amannthm},
there exists $C_{8}$ 
such that 
$y(t)\le C_{8}$ for all $t\in [0,\tau]$.
Together with \eqref{Linf},
we obtain the desired estimate.
\end{proof}

\begin{proof}[Proof of Theorem \ref{1stthm}]
Suppose by contradiction that $T_{{\rm max}}<\infty$.
Then \eqref{Amanncri} has to follow by Theorem \ref{Amannthm}.
On the other hand, Lemma \ref{W14lem} proves that
\eqref{Amanncri} never occurs under the assumptions of Theorem \ref{1stthm}.
This contradiction enables us to conclude 
$T_{{\rm max}}=\infty$.
Furthermore,
Lemma \ref{comlem} asserts that,
if $\widetilde{a}=\max\{a_{1}^{+},a_{2}^{+}\}=0$,
namely, $a_{1}$, $a_{2}\le 0$, then
all nonnegative solutions of \eqref{para}
asymptotically decay to zero
in the sense that
for any $\varepsilon>0$,
there exists 
$t_{0}=t_{0}(\varepsilon, \|u_{0}+\gamma v_{0}\|_{L^{\infty}(\Omega )})>0$ such that 
$$
t\ge t_{0}\ \Longrightarrow\ 
\|u(\,\cdot\,,t)\|_{L^{\infty}(\Omega )}+\gamma
\|v(\,\cdot\,,t)\|_{L^{\infty}(\Omega )}\le 
\varepsilon;$$
whereas if 
$\widetilde{a}=\max\{a_{1}^{+},a_{2}^{+}\}>0$,
then there exists
$t_{0}=t_{0}(\|u_{0}+\gamma v_{0}\|_{L^{\infty}(\Omega )})>0$ such that 
$$
t\ge t_{0}\ \Longrightarrow\ 
\|u(\,\cdot\,,t)\|_{L^{\infty}(\Omega )}+\gamma
\|v(\,\cdot\,,t)\|_{L^{\infty}(\Omega )}\le 
\dfrac{2(1+\gamma^{2})\widetilde{a}}{\lambda^{*}}.$$
Then \eqref{abs} holds true with
\begin{equation}\label{Kdef}
K=
\begin{cases}
\varepsilon \quad&\mbox{if}\ 
a_{1}\le 0\ \mbox{and}\ a_{2}\le 0,\\
\frac{2(1+\gamma^{2})\widetilde{a}}{\lambda^{*}}\quad&\mbox{if}\ 
a_{1}>0\ \mbox{or}\ a_{2}> 0.
\end{cases}
\end{equation}
The proof of Theorem \ref{1stthm} is complete.
\end{proof}

\section{Global asymptotic stability of the positive constant 
steady state}
In this section, we give a proof of Theorem \ref{2ndthm}.
In order to show the global asymptotic stability of 
the positive constant steady state $(u^{*}, v^{*})$ 
stated as \eqref{L2aym},
we construct a Lyapunov function.
To do so, we introduce a nonnegative function
$H(\eta; \xi )$ as follows:
$$
H(\eta; \xi):=\eta-\xi-\xi\log\dfrac{\eta}{\xi}
\qquad 
\mbox{for}\ \eta>0,\ \xi>0.
$$
Hence $(0,\infty)\ni\eta\mapsto H(\eta; \xi)\in [0,\infty )$
is monotone decreasing for $\eta\in (0,\xi)$,
attains the minimum zero at $\eta=\xi$
and
is monotone increasing for $\eta\in (\xi,\infty )$
with 
$\lim_{\eta\searrow 0}H(\eta; \xi)=
\lim_{\eta\to\infty}H(\eta; \xi)=\infty$.
Under the assumptions
of Theorem \ref{2ndthm},
let $(u,v)$ be the positive solution
of \eqref{para}
(which exists globally in time by Theorem \ref{1stthm}).
For the solution $(u,v)$, we define an energy functional 
$\mathcal{F}(t)$ as follows:
$$
\mathcal{F}(t):=
\int_{\Omega}
\{\,
H(u(x,t); u^{*})+\gamma
H(v(x,t); v^{*})\,\},
$$
where the positive constant steady state
$(u^{*}, v^{*})$ exists in the form \eqref{const}
when \eqref{weaker} and either of (i)-(iii) in \eqref{weakconst} 
are assumed.
We further define a function $\mathcal{D}(t)$ 
that measures the $L^{2}(\Omega )$ distance between 
$(u, v)$ and $(u^{*}, v^{*})$ as follows:
$$
\mathcal{D}(t):=
\|u(\,\cdot\,,t)-u^{*}\|_{L^{2}(\Omega )}^{2}+
\|v(\,\cdot\,,t)-v^{*}\|_{L^{2}(\Omega )}^{2}.
$$
It will be shown that 
the energy function $\mathcal{F}(t)$ is monotone decreasing
along the trajectory $(u(\,\cdot\,,t), v(\,\cdot\,,t))$
for $t>0$, and moreover,
the derivative $\mathcal{F}'(t)$ is
less than a negative multiple of the $L^{2}(\Omega )$ distance
$\mathcal{D}(t)$ between 
$(u(\,\cdot\,,t), v(\,\cdot\,,t))$ and
$(u^{*}, v^{*})$ if $t>0$ is sufficiently large.
This fact together with
the uniform continuity of $\mathcal{D}(t)$ ensured by 
Lemma \ref{Hollem} will show the convergence
\eqref{L2aym}.

\begin{lem}\label{dFlem}
Let $d_{1}=d_{2}\,(\,=:d\,)$ and
$\Omega$ be a bounded convex domain
in $\mathbb{R}^{N}$
with $N\le 3$.
Assume \eqref{weaker} and either of (i)-(iii) in \eqref{weakconst}.
Then, if 
$$
d\ge
\dfrac{(\beta u^{*}+\alpha v^{*}) K}
{2\sqrt{\gamma u^{*}v^{*}}},
$$
then any positive solution 
$(u, v)$ of 
\eqref{para}
fulfills
\begin{equation}\label{Ft}
\dfrac{d}{dt}\mathcal{F}(t)+\lambda^{*}\mathcal{D}(t)\le 0
\quad\mbox{for all}\ t\in [t_{0},\infty),
\end{equation}
where 
$(u^{*}, v^{*})$,
$\lambda^{*}$,
$K$ 
and 
$t_{0}$
are introduced by 
\eqref{const},
\eqref{lamdef},
\eqref{Kdef}
and 
\eqref{abs}, respectively.
\end{lem}

\begin{proof}
From the first equation of \eqref{para} and
the fact that $a_{1}-b_{1}u^{*}+c_{1}v^{*}=0$,
one can verify that any solution $(u,v)$ 
of \eqref{para}
satisfies
\begin{equation}\label{Hdu}
\begin{split}
\dfrac{d}{dt}
\int_{\Omega }
H(u(x,t); u^{*})
=&
\int_{\Omega}
\dfrac{\partial}{\partial t}
\biggl(
u-u^{*}-u^{*}\log\dfrac{u}{u^{*}}\biggr)
=\int_{\Omega}
\biggl(1-\dfrac{u^{*}}{u}\biggr)u_{t}\\
=&
-u^{*}\int_{\Omega}
(d+\alpha v)
\biggl|
\dfrac{\nabla u}{u}\biggr|^{2}
+
\alpha u^{*}\int_{\Omega}
\dfrac{\nabla u\cdot\nabla v}{u}+
\int_{\Omega }
(u-u^{*})(a_{1}-b_{1}u+c_{1}v)\\
=&
-u^{*}\int_{\Omega }(d+\alpha v)\biggl|
\dfrac{\nabla u}{u}\biggr|^{2}
+\alpha u^{*}\int_{\Omega}
\dfrac{\nabla u\cdot\nabla v}{u}\\
&-b_{1}\int_{\Omega}(u-u^{*})^{2}+
c_{1}\int_{\Omega }(u-u^{*})(v-v^{*})
\qquad\mbox{for all}\ t>0.
\end{split}
\end{equation}
Similarly, the second equation of \eqref{para} and
the fact that $a_{2}+b_{2}u^{*}-c_{2}v^{*}=0$ yield
\begin{equation}\label{Hdv}
\begin{split}
\dfrac{d}{dt}
\int_{\Omega }
H(v(x,t); v^{*})
=&
-v^{*}\int_{\Omega }(d+\beta u)\biggl|
\dfrac{\nabla v}{v}\biggr|^{2}
+\beta v^{*}\int_{\Omega}
\dfrac{\nabla u\cdot\nabla v}{v}\\
&-c_{2}\int_{\Omega}(v-v^{*})^{2}+
b_{2}\int_{\Omega }(u-u^{*})(v-v^{*})
\qquad\mbox{for all}\ t>0.
\end{split}
\end{equation}
By adding $\gamma$ times of \eqref{Hdv} to \eqref{Hdu}, 
we obtain the integral of quadratic forms as follows:
\begin{equation}\label{dF}
\dfrac{d}{dt}\underbrace{\int_{\Omega}
\{\,H(u(x,t), u^{*})+\gamma H(v(x,t), v^{*})\,\}}_{\mathcal{F}(t)}
=-\int_{\Omega}\boldsymbol{X}A\boldsymbol{X}^{T}
-\sum^{N}_{j=1}\int_{\Omega}\boldsymbol{Y}_{j}B\boldsymbol{Y}_{j}^{T},
\end{equation}
where
$$
\boldsymbol{X}=\boldsymbol{X}(x,t):=(u(x,t)-u^{*}, v(x,t)-v^{*}),\quad
\boldsymbol{Y}_{j}=\boldsymbol{Y}_{j}(x,t):=
\biggl(
\dfrac{1}{u}\biggl(\dfrac{\partial u}{\partial x_{j}}\biggr),
\dfrac{1}{v}\biggl(\dfrac{\partial v}{\partial x_{j}}\biggr)
\biggr),
$$
and
$$
B=B(x,t):=\left[
\begin{array}{cc}
u^{*}(d+\alpha v) & -\frac{\alpha (u^{*}v+v^{*}u )}{2}\\
-\frac{\alpha (u^{*}v+v^{*}u)}{2} & v^{*}(\gamma d+\alpha u)
\end{array}
\right]
$$
and
$A$ is the matrix with constant entries defined by \eqref{Adef}. 
In view of \eqref{lamdef}, 
we recall that $A$ is positive definite,
and therefore,
\begin{equation}\label{posiA}
\boldsymbol{X}A\boldsymbol{X}^{T}\ge \lambda^{*}
\{\,(u(x,t)-u^{*})^{2}+(v(x,t)-v^{*})^{2}\,\}
\quad\mbox{for all}\
(x,t)\in\Omega\times (0,\infty ).
\end{equation}
It follows from Theorem \ref{1stthm} and \eqref{Kdef} that
there exists $t_{0}=t_{0}(\|u_{0}+\gamma v_{0}\|_{L^{\infty}(\Omega )})>0$
such that if $t\ge t_{0}$, then
$$
\|u(\,\cdot\,,t)\|_{L^{\infty}(\Omega )}\le K
\quad\mbox{and}\quad
\|v(\,\cdot\,,t)\|_{L^{\infty}(\Omega )}\le \dfrac{\beta K}{\alpha },
$$
where $K$ is the positive number defined by \eqref{Kdef}.
It is noted that
\begin{equation}
\begin{split}
\det B&=u^{*}v^{*}(d+\alpha v)(\gamma d+\alpha u)-
\biggl(\dfrac{\alpha (u^{*}v+v^{*}u)}{2}\biggr)^{2}\\
&>
\gamma u^{*}v^{*}d^{2}-
\biggl(\dfrac{(\beta u^{*}+\alpha v^{*}) K}{2}\biggr)^{2}
\quad\mbox{for all}\ (x,t)\in\Omega\times [t_{0},\infty).
\end{split}
\nonumber
\end{equation}
Therefore, if
$$
d\ge \dfrac{(\beta u^{*}+\alpha v^{*}) K}{2\sqrt{\gamma u^{*}v^{*}}}
\,(\,=:\overline{d}\,),
$$
then
$B=B(x,t)$ is positive definite
for all $(x,t)\in\Omega\times [t_{0},\infty)$,
thereby
$$
\boldsymbol{Y}_{j}B\boldsymbol{Y}_{j}^{T}>0
\quad\mbox{for all}\ (x,t)\in\Omega\times [t_{0},\infty)
\quad
\mbox{and}\quad j=1,\ldots,N.
$$
By virtue of \eqref{dF} and \eqref{posiA}, we know that
if $d\ge\overline{d}$, then
$$
\dfrac{d}{dt}\mathcal{F}(t)+\lambda^{*}\mathcal{D}(t)\le 0
\quad\mbox{for all}\ t\in [t_{0},\infty).
$$
Then the proof of Lemma \ref{dFlem} is complete.
\end{proof}

\begin{proof}[Proof of Theorem \ref{2ndthm}]
By integrating \eqref{Ft} over $(t_{0},t)$, one can see
$$
\mathcal{F}(t)-\mathcal{F}(t_{0})+\lambda^{*}\int^{t}_{t_{0}}
\mathcal{D}(s)ds\le 0.
$$
Since $\mathcal{F}(t)>0$ for all $t\in [t_{0}, \infty)$, then 
$$
\int^{t}_{t_{0}}\mathcal{D}(s)\,ds<
\dfrac{\mathcal{F}(t_{0})}{\lambda^{*}}
\quad\mbox{for all}\ t\in (t_{0}, \infty).
$$
Hence
we set $t\to\infty$
in the above inequality to get
\begin{equation}\label{intD}
\int^{\infty}_{t_0}
\mathcal{D}(t)\,dt\le
\dfrac{\mathcal{F}(t_{0})}{\lambda^{*}}<\infty.
\end{equation}
We show $\mathcal{D}(t)\to 0$ as $t\to\infty$.
Suppose for contradiction that
$\limsup_{t\to\infty}D(t)>0$.
Then there exist a constant $\varepsilon>0$ and 
a sequence $\{t_{n}\}$ with $\lim_{n\to\infty}t_{n}=\infty$
such that
$$
\mathcal{D}(t_{n})>\varepsilon\quad\mbox{for all}\ n\in\mathbb{N}.
$$
In view of Lemma \ref{Hollem},
we recall that the constant $\widehat{C}$ in the H\"older estimate
\eqref{HolC} can be taken independently $t\in (\tau, \infty)$.
(Note that $\widehat{T}_{{\rm max}}=\infty$ under the assumption
of Theorem \ref{2ndthm}.)
This fact implies that
$\mathcal{D}(t)$ is uniformly continuous for $t\in [t_{0},\infty)$.
Therefore, for the above $\varepsilon>0$,
there exists $\delta>0$ such that
$$
|t-t_{n}|<\delta\ \Longrightarrow\
|\mathcal{D}(t)-\mathcal{D}(t_{n})|<\dfrac{\varepsilon}{2}
\quad\mbox{for all}\ 
t\in [t_{0}, \infty )\quad\mbox{and}\quad n\in\mathbb{N}.
$$
Then it follows that
$$
\mathcal{D}(t)\ge \mathcal{D}(t_{n})-|\mathcal{D}(t_{n})-\mathcal{D}(t)|
\ge \varepsilon-\dfrac{\varepsilon}{2}=\dfrac{\varepsilon }{2}
$$
if $t\in (t_{n},t_{n}+\delta)$ with each $n\in\mathbb{N}$.
Therefore, we know that
$$
\int^{t_{n}+\delta}_{t_{n}}
\mathcal{D}(t)\,dt\ge
\int^{t_{n}+\delta}_{t_{n}}
\dfrac{\varepsilon}{2}\,dt=
\dfrac{\delta\varepsilon}{2}>0
\quad\mbox{for all}\ n\in\mathbb{N},
$$
whereas
\eqref{intD} obviously leads to
$$
\int^{t_{n}+\delta}_{t_{n}}\mathcal{D}(t)\,dt\to 0
\quad\mbox{as}\ n\to\infty.
$$
This contradiction enables us to conclude that
$\mathcal{D}(t)\to 0$ as $t\to\infty$, namely,
$$
\lim_{t\to\infty}
(u(\,\cdot\,,t), v(\,\cdot\,,t))=(u^{*}, v^{*})\quad
\mbox{in}\ L^{2}(\Omega )\times L^{2}(\Omega ).
$$
Then we complete the proof of Theorem \ref{2ndthm}.
\end{proof}

\section{Relationship between stationary and nonstationary solutions}
In this section, we discuss the relationship between 
the bifurcation structure of the nonconstant 
positive steady states obtained in \cite{AK} and 
the time-global nonstationary solutions obtained 
up to the previous section.
In \cite{AK},
Adachi and the second author of the present paper studied
the corresponding stationary problem in the weak cooperative case
$b_{1}c_{2}>b_{2}c_{1}$.
Among other things, they showed that
nonconstant positive steady states bifurcate
from the positive constant steady state $(u^{*}, v^{*})$
when either of (ii) or (iii) of \eqref{weakconst} 
(under $b_{1}c_{2}>b_{2}c_{1}$) is satisfied
in addition to some conditions.
In [1], $d_i$ $(i=1\ \mbox{or}\ 2)$ is adopted as the bifurcation parameter, 
but in order to relate the results of \cite{AK} 
to the properties of the nonstationary solutions of \eqref{para} 
obtained up to the previous section, 
we state the results of \cite{AK} for the stationary problem
with equal random diffusion rates.
\begin{equation}\label{SP}
\begin{cases}
d\Delta u+\alpha\nabla\cdot\biggl[
v^{2}\,\nabla\biggl(\dfrac{u}{v}\biggr)\biggr]
+u(a_{1}-b_{1}u+c_{1}v)=0,\ \ \ 
&x\in\Omega,\vspace{1mm} \\
d\Delta v +\beta\nabla\cdot\biggl[
u^{2}\,\nabla\biggl(\dfrac{v}{u}\biggr)\biggr]
+v(a_{2}+b_{2}u-c_{2}v)=0,\ \ \ 
&x\in\Omega, \vspace{1mm} \\
\dfrac{\partial u}{\partial\nu}=
\dfrac{\partial v}{\partial\nu}=0,\ \ \ 
&x\in\partial\Omega
\end{cases}
\end{equation}  
with bifurcation parameter $d$, as Theorem \ref{AKthm1} below.
Since results in cases (ii) and (iii) of \eqref{weakconst}  
can be described similarly, 
the result in case (ii) will be discussed.
In the bifurcation structure, 
the eigenvalue problem:
\begin{equation}\label{Eg}
-\Delta \phi=\lambda \phi\quad\mbox{in}\ \Omega,\qquad
\dfrac{\partial \phi}{\partial\nu}=0
\quad\mbox{on}\ \partial\Omega
\end{equation}
will play an important role.
We denote all the eigenvalues by
$$0=\lambda_{0}<\lambda_{1}\le\lambda_{2}\le\cdots\le\lambda_{j}\le\cdots$$
counting multiplicity and
denote by
$\varPhi_{j}$ any $L^{2}(\Omega )$ normalized eigenfunciton
corresponding to the eigenvalue $\lambda_{j}$.
\begin{thm}[\cite{AK}]\label{AKthm1}
Assume that
\begin{equation}\label{hinpu}
a_{2}<0<a_{1},\quad
\dfrac{c_{1}}{c_{2}}<\dfrac{b_{1}}{b_{2}}<\dfrac{a_{1}}{|a_{2}|}
\end{equation}
and
$$
P_{j}(\alpha, \beta):=
\lambda_{j}(\alpha |a_{2}|v^{*}-\beta a_{1}u^{*})-
(a_{1}c_{2}-|a_{2}|c_{1})v^{*}>0$$
for some $j\in\mathbb{N}$.
If $\lambda_{j}$ is a simple eigenvalue, then
the curve of nonconstant solutions of \eqref{SP}
bifurcates from $(u^{*}, v^{*})$ at
\begin{equation}\label{dj}
d=d^{(j)}_{*}(\alpha, \beta ):=\dfrac{1}{2\lambda_{j}}
\left(
\sqrt{Q_{j}^{2}+4P_{j}}
-Q_{j}
\right),
\end{equation}
where
$
Q_{j}(\alpha, \beta):=\lambda_{j}(\beta u^{*}+\alpha v^{*})+b_{1}u^{*}+
c_{1}v^{*}$.
There exist a neighborhood $\mathcal{U}_{j}
\subset\mathbb{R}\times W^{2,p}_{\nu}(\Omega )^{2}$ of
$(d^{(j)}_{*}, u^{*}, v^{*})$ and
a small positive number $\delta_{j}$ such that
all solutions of \eqref{SP}
(treating $d$ as a positive parameter) contained in $\mathcal{U}_{j}$
consist of the union of
$\{(d,u^{*},v^{*})\in\mathcal{U}_{j}\}$
and
a simple curve
\begin{equation}\label{Gammaj}
\varGamma_{j}\,:\,
\left[
\begin{array}{c}
d\\
u\\
v
\end{array}
\right]=
\left[
\begin{array}{l}
d^{(j)}_{*}(\alpha, \beta)\\
u^{*}\\
v^{*}
\end{array}
\right]
+
\left[
\begin{array}{l}
q(s)\\
s(\varPhi_{j}+\widetilde{u}(s))\\
s(\kappa_{j}\varPhi_{j}+\widetilde{v}(s))\\
\end{array}
\right]
\quad\mbox{for}\ s\in (-\delta_{j}, \delta_{j})
\end{equation}
with some $\kappa_{j}>0$, where
$(q(s), \widetilde{u}(s), \widetilde{v}(s))
\in\mathbb{R}\times W^{2,p}_{\nu}(\Omega )^{2}$
is continuously differentiable for $s\in (-\delta_{j}, \delta_{j})$
with
$(q(0), \widetilde{u}(0), \widetilde{v}(0))
=(0,0,0)$ and
$\int_{\Omega}\varPhi_{j}\widetilde{u}(s)=
\int_{\Omega}\varPhi_{j}\widetilde{v}(s)=0
$ for $s\in (-\delta_{j}, \delta_{j})$.
\end{thm}
\begin{figure}
\begin{center}
{\includegraphics*[scale=.5]{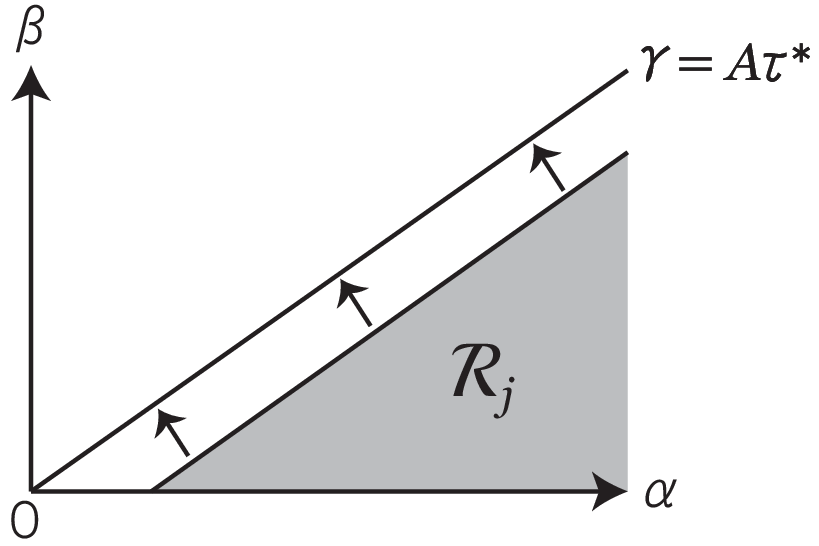}}\\
\caption{$\mathcal{R}_{j}$ and $\gamma =A\tau^{*}$ on the $\alpha\beta$ plane}
\end{center}
\end{figure}
From the viewpoint of the effect of the strongly coupled
diffusion terms on the bifurcation structure,
one can see that
if 
$$
(\alpha, \beta)\in\bigcup^{\infty}_{j=1}\mathcal{R}_{j},
\quad\mbox{where}\ \
\mathcal{R}_{j}:=
\{\,(\alpha, \beta )\in\mathbb{R}^{2}_{>0}\,:\,
P_{j}(\alpha, \beta )>0\,\},
$$
then there appear infinitely many bifurcation points
$(d_{*}^{(j)},u^{*}, v^{*})$ with 
$$
j\ge k:=\min\{\,j\in\mathbb{N}\,:\,
(\alpha, \beta)\in\mathcal{R}_{j}\,\}$$
such that $\lambda_{j}$ is a simple eigenvalue of
\eqref{Eg}.
Indeed, it is possible to check that 
$\mathcal{R}_{j}$ expands to the positive cone below the
line 
$$\gamma=A\tau^{*},\qquad\mbox{where}\ \ 
A=\dfrac{a_{1}}{|a_{2}|},\quad
\tau^{*}=\dfrac{u^{*}}{v^{*}},
$$
as $j\to\infty$ on the $\alpha\beta$ plane (see Figure 1).
Furthermore, from \eqref{dj},
one can verify that
$d^{(j)}_{*}(\alpha, \beta )\to 0$ as
$j\to\infty$ if $(\alpha, \beta)\in\cup^{\infty}_{j=1}
\mathcal{R}_{j}$.
This fact implies that $(\alpha, \beta )\in\cup^{\infty}_{j=1}
\mathcal{R}_{j}$ can produce various stationary patterns
when equal random diffusion $d$ is small enough.
In Figure 2, 
when $(a_{1}, a_{2}, b_{1}, b_{2}, c_{1}, c_{2})=
(1, -1, 4, 5, 2, 3)$
and 
$(\alpha, \beta )=(2,1)$, 
the numerical bifurcation diagram by
the continuation software pde2path (\cite{UWR})
is exhibited, where the horizontal axis represents $d$,
and the vertical axis represents $\|u\|_{L^{2}(\Omega )}$
(with $s\in (0,\delta_{j})$ in \eqref{Gammaj}).
The bifurcation point $d_{*}^{(j)}$ and 
the bifurcation branches $\varGamma_{j}$ of Theorem 5.1 can be 
observed within the numerically traceable range.
\begin{figure}
\begin{center}
{\includegraphics*[scale=.6]{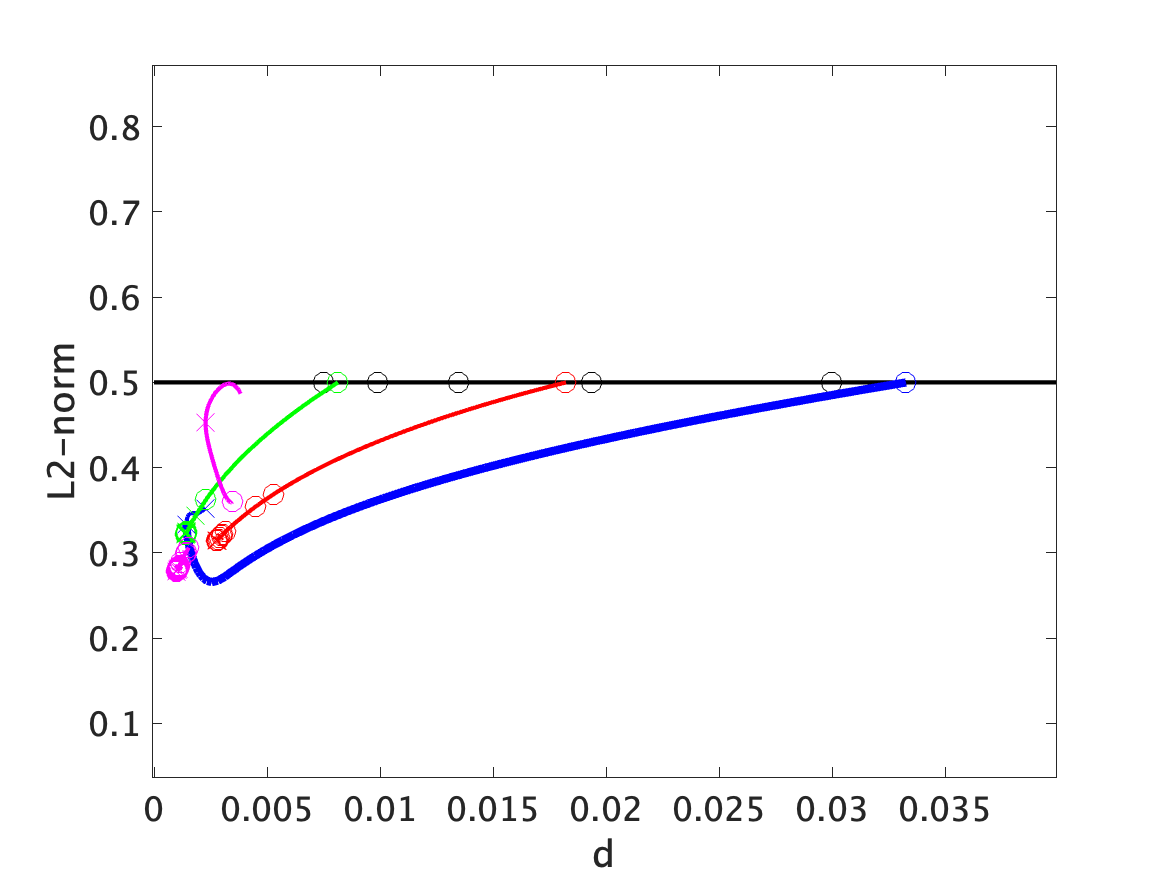}}\\
\caption{
Bifurcation branches of solutions to \eqref{SP} with bifurcation parameter $d$
in case $(\alpha, \beta, a_{1}, a_{2}, b_{1}, b_{2}, c_{1},c_{2})= 
(2,1,1,-1,4,5,2,3)$}
\end{center}
\end{figure}

According to \cite{AK},
in the nonlinear diffusion limit where
$\alpha$ and $\beta$ tend to infinity
keeping $(\alpha, \beta)\in\cup^{\infty}_{j=1}
\mathcal{R}_{j}$,
the positive steady states approach a positive steady state
to the scalar field equation as follows:
\begin{thm}[\cite{AK}]
Assume \eqref{hinpu} and $N\le 3$.
Let $\{(\alpha_{n}, \beta_{n})\}$ 
be any positive sequence satisfying
$\alpha_{n}\to\infty$,
$\beta_{n}\to\infty$
and
$\alpha_{n}/\beta_{n}\to\gamma$
with some $\gamma>A\tau^{*}$.
Let $\{(u_{n}, v_{n})\}$ be any sequence of positive nonconstant
solutions of \eqref{SP} with $(\alpha, \beta )=
(\alpha_{n}, 
\beta_{n})$.
Then there exists a positive function $v\in C^{2}(\overline{\Omega })$
such that
$$
\lim_{n\to\infty}(u_{n}, v_{n})=(\tau^{*}, 1)v
\quad\mbox{in}\ C^{1}(\overline{\Omega })\times
C^{1}(\overline{\Omega }),
$$
passing to a subsequence if necessary.
Furthermore, $v$ satisfies
\begin{equation}\label{sf}
\begin{cases}
d\Delta v+\xi^{*}v(v-v^{*})=0,
\qquad&x\in\Omega,\\
\dfrac{\partial v}{\partial\nu}=0,
\qquad&x\in\partial\Omega,
\end{cases}
\end{equation}
where
$$
\xi^{*}=\dfrac{\gamma |a_{2}|-\tau^{*}a_{1}}{u^{*}+\gamma v^{*}}.
$$
\end{thm}
It is well-known that the set of positive solutions
to the stationary scalar field equation such as \eqref{sf}
with small $d>0$ is very complicated from various viewpoints,
e.g. multiplicity, profiles and the bifurcation structure.
Naturally it can be expected that the spatio-temporal structure
of \eqref{para} with small equal random diffusion rates and the
weak cooperative setting is very rich
from the dynamical viewpoint
involving the bifurcation theory,
such a complicated dynamical picture of the strongly coupled
diffusion system is in stark contrast with that of
the weak cooperative linear diffusion system
($\alpha=\beta =0$ and $b_{1}c_{2}>b_{2}c_{1}$),
where $(u^{*}, v^{*})$ is globally asymptotically stable
attracting all the positive nonstationary solutions
as $t\to\infty$ regardless of the degree of 
the random diffusion rates.

Here in view of Theorem \ref{2ndthm},
we recall that if the equal random diffusion $d$ is sufficiently
large, then all positive solutions of \eqref{para} with
the weak cooperative setting tend to $(u^{*}, v^{*})$
as $t\to\infty$.
This scenario is same to the weak cooperative linear diffusion 
system.
Consequently, 
for the global dynamics of the weakly cooperative strongly coupled diffusion system \eqref{para}, we find that if 
the equal random diffusion $d$ is sufficiently large, 
all positive nonstationary solutions asymptotically approach
$(u^{*}, v^{*})$, 
as is the case for the corresponding linear diffusion system, 
whereas if $d$ is sufficiently small, 
a variety of dynamics can be expected where 
a lot of positive nonconstant steady states exist.


\end{document}